\documentclass[11pt]{article}%[11pt]

\usepackage{amsmath,amssymb,latexsym,amsfonts,mathrsfs,graphicx,verbatim}
\usepackage{algpseudocode,algorithm,algorithmicx}

\algrenewcommand\algorithmicrequire{\textbf{Precondition:}} 
\algrenewcommand\algorithmicensure{\textbf{Postcondition:}}
\algrenewcommand\alglinenumber[1]{\footnotesize #1} % To remove colon (:) from the line numbering in the algorithm
\algnewcommand{\IIf}[1]{\State\algorithmicif\ #1\ \algorithmicthen} %To include EndIf on the same line as If
\algnewcommand{\EndIIf}{\unskip\ \algorithmicend\ \algorithmicif}

\usepackage{graphicx}
\usepackage{setspace}
\usepackage{program}
\usepackage{subfig}
\usepackage{color}

\renewcommand{\quad}{$~~~\;\;\;$}
\newtheorem{theorem}{Theorem}

\newtheorem{proposition}{Proposition}
\newtheorem{lemma}{Lemma}
\newtheorem{corollary}{Corollary}
\newtheorem{example}{Example}

\newcommand{\diag}{\mbox{\rm diag}}
\newcommand{\tr}{\mbox{\rm trace}}

\newcommand{\transp}{{^{\rm T}}}

\newcommand{\vertiii}[1]{{\left\vert\kern-0.25ex\left\vert\kern-0.25ex\left\vert #1 
    \right\vert\kern-0.25ex\right\vert\kern-0.25ex\right\vert}}

\renewcommand{\R}{\mathbb{R}}
\renewcommand{\S}{\mathbb{S}}
\renewcommand{\L}{\mathbb{L}}
\newcommand{\matr}[1]{\begin{bmatrix} #1 \end{bmatrix}}    % matrix
\def\transp{^{\rm T}}
\newcommand{\Oh}{\mathcal O}

\newcommand{\ip}[2]{\left\langle #1 , #2 \right\rangle}    % inner product

\providecommand{\newoperator}[3]{%
  \newcommand*{#1}{\mathop{#2}#3}}

\newoperator{\argmax}{\mathsf{argmax}}{}
\newoperator{\argmin}{\mathsf{argmin}}{}
\newcommand{\1}{\mathbf{1}}
\newcommand{\dmin}{\displaystyle\min}
\newcommand{\dmax}{\displaystyle\max}

\newcommand{\0}{\mathbf{0}}

\usepackage[margin=3.75cm]{geometry}

\author{Javier Pe\~na\thanks{Tepper School of Business,
Carnegie Mellon University, USA, {\tt jfp@andrew.cmu.edu}}
\and  Negar Soheili\thanks{College of Business Administration,  University of Illinois at Chicago, USA, {\tt nazad@uic.edu }}
}
\title{Solving Conic Systems via Projection and Rescaling}

\begin{document}
\maketitle
\begin{abstract}
We propose a simple {\em projection and rescaling algorithm} to solve the feasibility problem
\[
\text{ find } x \in L \cap \Omega,
\]
where $L$ and $\Omega$ are respectively a linear subspace and the interior of a symmetric cone in a finite-dimensional vector space $V$.

This projection and rescaling algorithm is inspired by previous work on rescaled versions of the perceptron algorithm and by Chubanov's projection-based method for linear feasibility problems.  As in these predecessors, each main iteration of our algorithm contains two steps: a {\em basic procedure} and a {\em rescaling} step.  When $L \cap \Omega \ne \emptyset$, the projection and rescaling algorithm finds a point $x \in L  \cap  \Omega$ in at most $\Oh(\log(1/\delta(L \cap \Omega)))$ iterations,  where $\delta(L \cap \Omega) \in (0,1]$ is a measure of the most interior point in $L \cap  \Omega$. The ideal value $\delta(L\cap \Omega) = 1$ is attained when $L \cap \Omega$ contains the center of the symmetric cone $\Omega$.

We describe several possible implementations for the basic procedure including a  perceptron scheme and a smooth perceptron scheme.  The perceptron scheme  requires $\Oh(r^4)$ perceptron updates and the smooth perceptron scheme requires $\Oh(r^2)$ smooth perceptron updates, where $r$ stands for the Jordan algebra rank of $V$.
\end{abstract}  

\section{Introduction}

We propose a simple algorithm based on projection and rescaling operations to solve the feasibility problem 
\begin{equation}\label{generalprob}
\text{ find } x\in L\cap \Omega,
\end{equation}
where $L$ and $\Omega$ are respectively a linear subspace and the interior of a symmetric cone in a finite-dimensional vector space $V$.  Problem~\eqref{generalprob} is fundamental in optimization as it encompasses a large class of feasibility problems. 
For example, for $A\in \R^{m\times n}$ and $b\in \R^m,$ the problem $Ax = b, x> 0$ can be formulated as \eqref{generalprob} by taking $L = \{(x,t)\in\R^{n+1}: Ax -t b = 0\}$ and $\Omega = \R^{n+1}_{++}$.  For $A\in \R^{m\times n}, \; c\in \R^n,$ the problem $A\transp y < c$ can be formulated as \eqref{generalprob} by taking 
$L = \{(s,t)\in\R^{n+1}: tc-s \in \text{span}(A)\}$ and $\Omega = \R^{n+1}_{++}$.
Likewise, the strict semi-definite feasibility problem $AX = b, X\in\S^n_{++}$ 
can be formulated as \eqref{generalprob} by taking $L = \{(X,t)\in \S^n\times \R: AX-tb=0\}$ and $\Omega = \S^n_{++}\times \R_{++}$. 
 The problem of finding an $\epsilon$-solution to a primal-dual pair of conic optimization problems satisfying the Slater condition can also be recast as a problem of the form~\eqref{generalprob} via a similar type of homogenization. 
 
To solve~\eqref{generalprob}, we consider the equivalent problem
\begin{equation}\label{projectionprob}
\text{ find $z\in V$ such that } P_L z \in\Omega,\end{equation}
where $P_L: V\rightarrow V$ denotes the orthogonal projection onto the subspace $L$. Observe that if $z$ is a solution to~\eqref{projectionprob}, then $x = P_Lz$ is a  solution to~\eqref{generalprob}.  Conversely, if $x$ is a solution to~\eqref{generalprob}, then $z=x$ is a solution to \eqref{projectionprob}.

Our {\em projection and rescaling algorithm} for \eqref{projectionprob} formalizes the following two intuitive ideas.  First, if the set $L\cap \Omega$ is {\em well-conditioned} in the sense that the subspace $L$ contains points well in the interior of $\Omega,$ then a {\em basic procedure,} which relies only on the projection mapping $P_L$, can easily find a point in $L\cap \Omega$. Second, when the basic procedure does not find a point in $L\cap \Omega$ after some amount of work, information about the problem instance can be inferred so that some type of {\em rescaling step} can be applied to obtain a better conditioned problem.  This two-step procedure eventually terminates with a feasible point in $L\cap \Omega$ provided this set is nonempty.

Our projection and rescaling algorithm is inspired by previous work on rescaled versions of the perceptron algorithm~\cite{BellFV09,DunaV06,PenaS13} as well as by Chubanov's work on a projection-based algorithm for linear feasibility problems~\cite{Chub15}.  In particular, 
the article~\cite{PenaS13} is concerned with a  feasibility problem of the form
\begin{equation}\label{alternative}
\text{ find }  y \in F,
\end{equation}
where $F \subseteq W$ is an open convex cone in a finite dimensional vector space $W$, and it is only assumed that a separation oracle for $F$ is available.  The gist of the approach in~\cite{PenaS13} is to enhance a simple relaxation-type algorithm for \eqref{alternative},  namely the perceptron method, with a periodic rescaling of the ambient space $W$.  When the set $F$ is {\em well-conditioned} in the sense that the volume of $F\cap \{y \in \R^m: \|y\|_2 =1\}$ exceeds a certain minimum threshold, the perceptron algorithm can easily find a point in $F$.  When that is not the case, the perceptron algorithm identifies a direction $d$ in the ambient space $W$ such that a dilation along $d$ increases the volume of $F\cap \{y \in \R^m: \|y\|_2 =1\}$ by a constant factor.  We note that the article~\cite{PenaS13} was preceded and inspired by the work of Dunagan and Vempala~\cite{DunaV06} and  Belloni, Freund, and Vempala~\cite{BellFV09}, who introduced random rescaling as a technique for enhancing the perceptron algorithm.

Our projection and rescaling algorithm can be seen as an extension of the recent work of Chubanov~\cite{Chub15} for the feasibility problem
\begin{equation}\label{lp.prob}
\text{ find } \; x > 0 \text{ such that } Ax = 0,
\end{equation}
where $A\in\R^{m\times n}.$  Observe that \eqref{lp.prob} is a special case of \eqref{generalprob} for $L = \ker(A)$ and $\Omega = \R^n_{++}.$  Chubanov~\cite{Chub15} relies on the equivalent problem
\begin{equation}\label{lp.projection}
\text{ find $z\in \R^n$ such that } P_L z > 0,
\end{equation}
where $P_L$ denotes the orthogonal projection onto $L = \ker(A)$.  Chubanov~\cite{Chub15} proposes an algorithm that combines a {\em basic procedure} (a relaxation-type algorithm) for \eqref{lp.projection} 
with a  periodic rescaling of the ambient space $\R^n$.  When the set $\{x>0: Ax = 0\}$ is {\em well-conditioned} in the sense that there exists a point in $\{x>0: Ax = 0, \|x\|_\infty =1\}$ whose coordinates are bounded away from zero (for example when $\{x: Ax = 0, \frac{\1}{2} \le x \le \1\}\ne\emptyset$)
the basic procedure  easily finds a solution to \eqref{lp.projection}. 
 When the basic procedure does not easily find a solution, it identifies a coordinate $i$ such that every point in $\{x>0: Ax = 0, \|x\|_\infty = 1\}$ satisfies $x_i < 1/2$.  Hence a dilation of the ambient space $\R^n$ along the $i$-th coordinate transforms the set $\{x>0: Ax = 0\}$ into a set that is better conditioned.   Chubanov shows that when $A$ has rational entries, the resulting algorithm either finds a solution to \eqref{lp.prob} or concludes that~\eqref{lp.prob} is infeasible in a total number of operations that is polynomial in the bit-length representation of $A$.  The article by Chubanov~\cite{Chub15} is  similar in spirit to his previous article~\cite{Chub12}.  Like~\cite{PenaS13} and its predecessors~\cite{BellFV09,DunaV06}, both \cite{Chub15} and \cite{Chub12}, as well as this paper, can be seen as enhancements of the classical relaxation method~\cite{Agmo54,MotzS54}.  Chubanov's work has also been revisited and extended by various sets of authors~\cite{BasuDLJ13,LiRT15,Roos15}. The numerical experiments reported in the articles by Roos~\cite{Roos15} and by Li, Roos, and Terlaky~\cite{LiRT15} provide promising evidence of the computational effectiveness of Chubanov's method~\cite{Chub15} and related variants.

In a similar fashion to the approaches in~\cite{Chub15} and in~\cite{PenaS13}, 
we propose an algorithm for~\eqref{projectionprob} that combines a simple basic procedure with a periodic rescaling of $V$.  The analysis of our approach relies on a suitable {\em condition measure} $\delta(L\cap \Omega)\in (0,1]$ associated to the {\em most interior} point in $L\cap \Omega$.  The ideal value $\delta(L\cap \Omega) = 1$ is attained when $L\cap \Omega$ contains the center of the cone  $\Omega.$
The main steps in our projection and rescaling algorithm can be sketched as follows.  When $\delta(L\cap \Omega)$ exceeds a certain threshold, a {\em basic procedure} easily finds a point $z \in L\cap \Omega$.  On the other hand, when that is not the case, the basic procedure identifies a linear automorphism $D:V\rightarrow V$ that leaves $\Omega$ unchanged and such that $\delta(D(L)\cap \Omega) > 1.5 \cdot \delta(L\cap \Omega).$  
The algorithm then continues with the  transformed problem 
\[
\text{ find } \; x \in D(L) \cap \Omega.
\]  
 As Theorem~\ref{main.thm} below formally shows, if $L\cap \Omega \ne \emptyset$ then the projection and rescaling algorithm finds a point in $L\cap \Omega$ after $\Oh(\log(1/\delta(L\cap \Omega))$ rounds of this combination of basic procedure and rescaling step.

We describe several \emph{elementary} implementations for the basic procedure including a perceptron scheme~\cite{BellFV09,Rosen}, a von Neumann scheme~\cite{EpelF00}, and variations of each of them, namely a von Neumann scheme with away~steps~\cite{PenaRS15}, and a smooth perceptron scheme~\cite{SohPena,SoheP13}. A common attractive feature of all of these schemes is their low computational work per iteration. We show that the first three schemes require $\Oh(r^4)$ simple updates and the smooth perceptron algorithm requires  $\Oh(r^2)$ simple updates, where $r$ is the Jordan algebra rank of $V$.   In the special case $\Omega = \R^n_{++},$ we have $r=n$ but the first three schemes require $\Oh(n^3)$ simple updates and the smooth perceptron scheme requires $\Oh(n^{3/2})$ simple updates.

It is worth noting that the problems \eqref{generalprob} and \eqref{alternative} are alternative systems when $F = \{y: A^*y \in \Omega^*\}$ for a linear mapping $A: V\rightarrow W$ with $L=\ker(A)$.  In this case the rescaling operation in~\cite{PenaS13} can be seen as a type of {\em left reconditioning} that transforms $A$ to  $D A$ for some isomorphism $D:W\rightarrow W$.  On the other hand, the rescaling operation in~\cite{Chub15} and its general version in this paper can be seen as a type of {\em right reconditioning} that transforms $A$ to  $A D$ for some isomorphism $D: V\rightarrow V$ that satisfies $D(\Omega) = \Omega$.  These kinds of left and right reconditioning operations are in the same spirit as the left and right preconditioners  operations introduced and discussed in~\cite{PenaRS14}.   

Observe that the reformulations~\eqref{projectionprob} and \eqref{lp.projection} are amenable to the algorithmic scheme developed in \cite{PenaS13} since they are of the form~\eqref{alternative}.  However, the algorithmic scheme in \cite{PenaS13} relies solely on separation and hence does not take advantage of the properties of the symmetric cone $\Omega$.  Not surprisingly, the algorithmic scheme in~\cite{PenaS13} applied to~\eqref{projectionprob} could be weaker than the one presented in this paper. In particular, the iteration bound for the {\em perceptron phase} in the algorithmic scheme in \cite{PenaS13} applied to \eqref{projectionprob}  depends on the dimension of the vector space $V$.   By contrast, the iteration bound for the {\em basic procedure} of the algorithm in this paper depends on the Jordan algebra rank of $V$ which is at most equal to the dimension of $V$ but could be quite a bit smaller.  For instance, the Jordan algebra rank of $\S^n$ is $n$ whereas its dimension is $n(n+1)/2$.  If $V$ is endowed with the Jordan algebra associated to the second-order cone, then its Jordan algebra rank is only 2 regardless of its dimension.

The main sections of the paper are organized as follows. In Section~\ref{sec.chubanov} we describe a Projection and Rescaling Algorithm that is nearly identical to that proposed by Chubanov~\cite{Chub15} for the special case of problem~\eqref{generalprob} when $V= \R^n$ and $\Omega = \R^n_{++}$, albeit presented in a slightly different format.  The main purpose of this section is to introduce the algorithmic scheme and main ideas that we subsequently generalize.  In Section~\ref{sec.SDP} we extend our Projection and Rescaling Algorithm to the case when $V$ is the space $\S^n$ of symmetric $n\times n$ real matrices and $\Omega$ is the cone $\S^n_{++}$ of positive definite matrices.  This is a special but particularly important case of the more general case when $V$ is a vector space endowed with an Euclidean Jordan algebra structure and $\Omega$ is the interior of the cone of squares in $V$, which is presented in Section~\ref{sec.symmetric}. In Section~\ref{sec.basic.proc} we describe different implementations for the basic procedure.  Finally in Section~\ref{sec.projection} we discuss how the projection matrix in~\eqref{projectionprob} can be updated after each rescaling operation.

\section{Projection and rescaling algorithm}
\label{sec.chubanov}
Assume $L \subseteq \R^n$ and consider the problem 
\begin{equation}\label{lp.problem}
\text{ find } x \in L \cap \R^n_{++}.
\end{equation}
Let $P_L:\R^n \rightarrow \R^n$ be the projection onto $L$.  Then \eqref{lp.problem} is equivalent to
\[
\text{ find } y \in \R^n \text{ such that } P_Ly \in \R^n_{++}.
\]
Consider the following kind of {\em condition measure} of the set $L \cap \R^n_{++}$:
\[
\delta(L \cap \R^n_{++}) := \dmax_x\left\{\prod_{i=1}^n x_i: x\in L \cap \R^n_{++}, \|x\|^2_2 = n\right\}. 
\]
Observe that $\delta(L\cap \R^n_{++}) >0 $ provided $L \cap \R^n_{++}\ne \emptyset.$ Furthermore, by the arithmetic-geometric inequality, $\delta(L\cap \R^n_{++}) \le 1$, and $\delta(L\cap \R^n_{++}) = 1$  precisely when $ e = \matr{1 & \cdots & 1}\transp \in L \cap \R^{n}_{++}.$  

For $v\in\R^n$ let $v^+$ denote $\max(v, 0)$ componentwise. Let $e_i\in \R^n$ denote the unitary vector whose $i$-th entry is equal to one and all others are equal to zero. The following key observation suggests a certain rescaling as a reconditioning operation.

\begin{proposition}\label{prop.chubanov}  Assume $z \in\R^n_{+}\setminus \{0\}$ is such that $\|(P_Lz)^+\|_2 \le \epsilon \|z\|_\infty$ for some $\epsilon \in (0,1).$ Let $D = I + a e_i e_i\transp$ where $i$ is such that $z_i = \dmax_{j=1,\dots,n} z_j$ and $a > 0$.  Then
\[
\delta(DL \cap \R^n_{++}) \ge \frac{1+a}{(1+(2a+a^2)\epsilon^2)^{n/2}}\cdot\delta(L \cap \R^n_{++}).
\]
In particular, if $\epsilon = \frac{1}{3\sqrt{n}}$ and $a=1$ then
\[
\delta(DL \cap \R^n_{++}) \ge 1.5\cdot\delta(L \cap \R^n_{++}).
\]\end{proposition}
\begin{proof}
Observe that for $x \in L\cap \R^n_{++}$ the point $\hat x:=\frac{\sqrt{n}}{\|Dx\|_2}Dx$ satisfies $\hat x \in DL \cap \R^n_{++}$ and $\|\hat x\|_2^2 = n$.
Thus it suffices to show that for $x \in L\cap \R^n_{++}$ with $\|x\|_2^2 = n$ both $\prod_{j=1}^n (Dx)_j = (1+a) \prod_{j=1}^n x_j$ and $\|Dx\|_2^2 \le n(1+(2a+a^2)\epsilon^2)$ as this would imply $\prod_{j=1}^n (\hat x)_j  \ge \frac{1+a}{(1+(2a+a^2)\epsilon^2)^{n/2}}\cdot  \prod_{j=1}^n x_j.$ Equivalently, it suffices to show that for $x\in L \cap \R^n_{++}$ with $\|x\|_2 = 1$ we have $\prod_{j=1}^n (Dx)_j \geq (1+a) \prod_{j=1}^n x_j $ and $\|Dx\|_2^2 \le 1+(2a+a^2)\epsilon^2$. 

Assume $x\in L \cap \R^n_{++}$ with $\|x\|_2 = 1$ is fixed. Since $Dx = (I+ae_ie_i\transp)x = x + ax_ie_i$, we have \[\prod_{j=1}^n (Dx)_j {\color{blue}{=}} (1+a) \prod_{j=1}^n x_j 
\]
Furthermore, since $x\in L \cap \R^n_{++}, \; \|x\|_2 = 1,$ and $z \ge 0$ it follows that \[0 < x_iz_i \le x\transp z = x\transp P_Lz \le \|\left(P_Lz\right)^+\|_2 \leq \epsilon z_i.\] Hence $x_i\le \epsilon$ and so $\|Dx\|_2^2 = \|x\|_2^2+(2a+a^2)x_i^2 \le 1 + (2a+a^2)\epsilon^2.$
\qed

\end{proof}

\bigskip

Proposition~\ref{prop.chubanov} suggests the Projection and Rescaling Algorithm described in Algorithm~\ref{alg:chubanov} below.  We note that Algorithm~\ref{alg:chubanov} is nearly identical to the algorithm proposed by Chubanov~\cite{Chub15}, albeit presented in a slightly different format.

\bigskip

\begin{algorithm}
  \caption{Projection and Rescaling Algorithm
    \label{alg:chubanov}}
  \begin{algorithmic}[1]
    \State  ({\bf Initialization})
  
    \Statex  Let $P_L \in \R^{n\times n}$ be the orthogonal projection onto $L$. 
    \Statex Let $D := I$ and $P := P_L.$ 

	\State ({\bf Basic Procedure})
	\Statex 
	Find  $z \gneqq 0$ such that either $Pz > 0$ or $\|(Pz)^+\|_2 \le \dfrac{1}{3\sqrt{n}} \|z\|_\infty$.
	\IIf{$Pz > 0$} HALT and \Return $x = D^{-1}Pz \in L \cap \R^n_{++}$ \EndIIf  
  	\State ({\bf Rescaling step}) 
	\Statex Pick $i$ such that $z_i = \|z\|_\infty$.
	\Statex Put $D:=(I+e_ie_i\transp)D$ and $P:=P_{DL}.$ 
	\Statex Go back to step 2. 
  \end{algorithmic}
\end{algorithm}

Theorem~\ref{thm.chubanov} states the main property of the above algorithm.  A major difference from the results in \cite{Chub15} is that Theorem~\ref{thm.chubanov} depends solely on $\delta(L\cap \R^n_{++})$.  In particular, $L$ can be any arbitrary linear subspace of $\R^n$.  It is not necessarily assumed to be the null space of a matrix with rational entries.  

\begin{theorem}\label{thm.chubanov} If $L \cap \R^n_{++} \ne \emptyset$ then Algorithm~\ref{alg:chubanov} finds $x\in L \cap \R^n_{++} $ in at most $\log_{1.5}(1/\delta(L\cap \R^n_{++}))$ main iterations. 
\end{theorem}
\begin{proof}  This is an immediate consequence of Proposition~\ref{prop.chubanov} and the fact that $\delta(\tilde L\cap \R^n_{++}) \le 1$ for any linear subspace $\tilde L \subseteq \R^n$ with $\tilde L \cap \R^n_{++} \ne \emptyset$.
\qed
\end{proof}

\medskip

To complement the statement of Theorem~\ref{thm.chubanov}, we next account for the number of arithmetic operations required by Algorithm~\ref{alg:chubanov}.  
A call to the basic procedure is the bulk of the computational work in each main iteration of  Algorithm~\ref{alg:chubanov}.  As we discuss in detail in Section~\ref{sec.basic.proc}, there are several possible implementations for the basic procedure.  The simplest implementations for the basic procedure terminate $\Oh(n^3)$ perceptron or von Neumann steps. Each of these steps requires a matrix-vector multiplication of the form $z\mapsto Pz$ in addition to some other negligible operations.  As we explain in Section~\ref{sec.projection} below, the projection matrix $P$ can be stored and updated in the form $P=QQ\transp$ for some matrix $Q\in \R^{n\times m}$ where $m=\dim(L)$ and the columns of $Q$ form an orthogonal basis of $DL$.  For a matrix of this form, each matrix-vector multiplication $z\mapsto Pz$ requires $\Oh(mn)$ arithmetic operations.  It thus follows that the total number of arithmetic operations required by Algorithm~\ref{alg:chubanov} is bounded above by $$\Oh(mn \cdot n^3\cdot \log(1/\delta(L\cap \R^n_{++}))) = \Oh(mn^4\log(1/\delta(L\cap \R^n_{++}))).$$

\medskip

Algorithm~\ref{alg:chubanov} is designed to find a solution to~\eqref{lp.problem} assuming that $L\cap\R^n_{++} \ne \emptyset$. If $L\cap\R^n_{++} = \emptyset$, Algorithm~\ref{alg:chubanov} will not terminate. However, Algorithm~\ref{alg:chubanov} has the 
straightforward extension described as Algorithm~\ref{alg:extended} that solves either~\eqref{lp.problem} or its strict alternative
$$ \text{find} \; \hat x \in L^{\perp}\cap\R^n_{++}$$
provided at least one of them is feasible.  An immediate consequence of Theorem~\ref{thm.chubanov} is that Algorithm~\ref{alg:extended} will find either $x\in L\cap \R^n_{++}$ or $\hat x \in L^{\perp}\cap\R^n_{++}$ in at most $\log_{1.5}(1/\max(\delta(L\cap \R^n_{++}),\delta(L^{\perp}\cap \R^n_{++})))$ main iterations provided $L\cap \R^n_{++}\ne \emptyset$ or $L^{\perp}\cap \R^n_{++}\ne \emptyset$.

\begin{algorithm}
  \caption{Extended Projection and Rescaling Algorithm
    \label{alg:extended}}
  \begin{algorithmic}[1]
    \State  ({\bf Initialization})
  
    \Statex  Let $P_L \in \R^{n\times n}$  be the orthogonal projection onto $L$. 
\Statex
Let $P_{L^\perp} \in \R^{n\times n}$     be the orthogonal projection onto      $L^{\perp}$. 
    \Statex Let $D := I$ and $P := P_L$. 
       \Statex Let $\hat D := I$ and $\hat P := P_{L^\perp}$.  

	\State ({\bf Basic Procedure})
	\Statex 
	Find  $z \gneqq 0$ such that either $Pz > 0$ or $\|(Pz)^+\|_2 \le \dfrac{1}{3\sqrt{n}} \|z\|_\infty$.
	\Statex
	Find  $\hat z \gneqq 0$ such that either $\hat P\hat z > 0$ or $\|(\hat P\hat z)^+\|_2 \le \dfrac{1}{3\sqrt{n}} \|\hat z\|_\infty$.

	\IIf{$Pz > 0$} HALT and \Return $x = D^{-1}Pz \in L \cap \R^n_{++}$ \EndIIf  
	\IIf{$\hat P\hat z > 0$} HALT and \Return $\hat x = \hat D^{-1}\hat P\hat z \in L^{\perp} \cap \R^n_{++}$ \EndIIf  

  	\State ({\bf Rescaling step}) 
	\Statex Pick $i$ such that $z_i = \|z\|_\infty$.
	\Statex Put $D:=(I+e_ie_i\transp)D$ and $P:=P_{DL}$. 
	\Statex Pick $j$ such that $\hat z_j = \|\hat z\|_\infty$.
	\Statex Put $\hat D:=(I+e_je_j\transp)\hat D$ and $P:=P_{\hat D L^{\perp}}$. 
	\Statex Go back to step 2.
  \end{algorithmic}
\end{algorithm}

\medskip

We conclude this section by noting that the stopping condition $\|(Pz)^+\|_2 \le \frac{1}{3\sqrt{n}} \|z\|_\infty$ in the basic procedure can be replaced by the less stringent condition 
$\|(Pz)^+\|_1 \le \frac{1}{2} \|z\|_\infty$.  This is closer to the approach used by Chubanov~\cite{Chub15}.   With this substitution it follows that  if $L \cap \R^n_{++} \ne \emptyset$ then the algorithm finds $x\in L \cap \R^n_{++} $ in at most $\log_{2}(1/\delta_{\infty}(L\cap \R^n_{++}))$ main iterations, where
\[
\delta_{\infty}(L \cap \R^n_{++}) := \dmax_x\left\{\prod_{i=1}^n x_i: x\in L \cap \R^n_{++}, \|x\|_{\infty} = 1\right\}. 
\]
We chose to state the above Projection and Rescaling Algorithm with the stopping condition  $\|(Pz)^+\|_2 \le \frac{1}{3\sqrt{n}} \|z\|_\infty$  and presented the above statements in terms of $\delta(L\cap \R^n_{++})$ because this approach has a more natural extension to symmetric cones.

\section{Extension to semidefinite conic systems}
\label{sec.SDP}
Let $\S^n$ denote the space of $n\times n$ real symmetric matrices. 
Assume $L\subseteq \S^n$ is a linear subspace and consider the problem 
\begin{equation}\label{sdp.problem}
\text{ find } X \in L \cap \S^n_{++},
\end{equation}
where $\S^n_{++}$ is the set of positive definite matrices, that is, the interior of the cone $\S^n_{+}$ of positive semidefinite matrices.

Assume the space $\S^n$ is endowed with the trace inner product 
\[
X\bullet S=\ip{X}{S} := \tr(XS).
\]
Let $P_L:\S^n \rightarrow \S^n$ be the projection onto $L$ with respect to the trace inner product.  Then \eqref{sdp.problem} is equivalent to
\[
\text{ find } Y \in \S^n \text{ such that } P_LY \in \S^n_{++}.
\]
For $X \in \S^n$ let $\lambda(X) \in \R^n$ denote the vector of eigenvalues of $X$.
We will rely on the Frobenius norm $\|X\|_F:=\sqrt{\ip{X}{X}} = \|\lambda(X)\|_2$ as well as on the operator norm $\|X\| := \dmax_{\|u\|_2 = 1} \|Xu\|_2 = \|\lambda(X)\|_\infty$.

Consider the following kind of {\em condition measure} of the set $L \cap \S^n_{++}$:
\[
\delta(L \cap \S^n_{++}) := \dmax_X\{\det(X): X\in L \cap \S^n_{++}, \|X\|_F^2 = n\}. 
\]  
In analogy to the case discussed in the previous section, $L\cap \S^n_{++}\ne \emptyset$ implies that $\delta(L \cap \S^n_{++}) \in (0,1]$ and $\delta(L \cap \S^n_{++}) = 1$ precisely when $I \in L \cap \S^n_{++}$.

For $X \in \S^n$ let $X^+$ denote the projection of $X$ on $\S^n_+$.  It is known, and easy to show, that if $X = Q\diag(\lambda(X))Q\transp$ is the spectral decomposition of $X$ then $X^+ = Q\diag(\lambda(X)^+)Q\transp$.

The key property stated as Proposition~\ref{prop.chubanov} above extends as follows.
\begin{proposition}  Assume $Z \succneqq 0$ is such that $\|(P_LZ)^+\|_F \le \epsilon \|Z\|.$  Let $u \in \R^n, \|u\|_2  =1$ be an eigenvector of $Z$ with eigenvalue $\lambda_{\max}(Z) = \|Z\|$. Let $D: \S^n \rightarrow \S^n$ be defined as
$$ D(X) = (I + a uu\transp)X(I + auu\transp)$$ 
for some constant $a>0$.   Then
\[
\delta(D(L) \cap \S^n_{++}) \ge \frac{(1+a)^2}{\left(1+(2a+a^2)\epsilon\right)^n} \cdot\delta(L \cap \S^n_{++}).
\]
In particular, if $\epsilon = \frac{1}{4n}$ and $a = \sqrt{2}-1$ then
\[
\delta(D(L) \cap \S^n_{++}) \ge 1.5 \cdot \delta(L \cap \S^n_{++}).
\]
\end{proposition}
\begin{proof}
Observe that for $X \in L\cap \S^n_{++}$ the point $\hat X:=\frac{\sqrt{n}}{\|D(X)\|_F}\cdot D(X)$ satisfies $\hat X \in D(L) \cap \S^n_{++}$ and $\|\hat X\|_F^2 = n$.
Thus it suffices to show that for $X \in L\cap \S^n_{++}$ with $\|X\|_F^2 = n$ both $\det(D(X)) = (1+a)^2 \det(X)$ and $\|D(X)\|_F^2 \le n(1+(2a+a^2)\epsilon)^2$ as this would imply $\det(\hat X)  \ge \frac{(1+a)^2}{(1+(2a+a^2)\epsilon^2)^{n}}\cdot  \det(X).$   Equivalently, it suffices to show that for $X \in L\cap \S^n_{++}$ with $\|X\|_F = 1$ both $\det(D(X)) = (1+a)^2 \det(X)$ and $\|D(X)\|_F^2 \le (1+(2a+a^2)\epsilon)^2.$

Assume $X \in L\cap \S^n_{++}$ with $\|X\|_F = 1$ is fixed.  Since $D(X) = (I + a uu\transp)X(I + auu\transp)$ and $\|u\|_2 = 1$, it readily follows that 
\[\det(D(X)) = \det(I+a uu\transp)^2\det(X) = (1+a)^2 \det(X).\]
The first step above holds because $\det(AB) = \det(A)\det(B)$ for all  $A,B\in\S^n$.  The second step holds because $\det(I+auu\transp) = 1+au\transp u = 1+ a$. 

On the other hand,
\begin{align}\label{frob.norm}
\|D(X)\|_F^2 &= \tr(D(X)^2)\notag \\  
&= \tr(X(I+auu\transp)^2X(I+ auu\transp)^2)\\
&= \tr(X^2+ 2(2a+a^2)X^2uu\transp + (2a+a^2)^2(u\transp Xu)Xuu\transp)\notag\\ &= \tr(X^2) + 2(2a+a^2)\tr(X^2 uu\transp )+ (2a+a^2)^2(u\transp X u)\tr(Xu u\transp)\notag\\ &=\tr(X^2) + 2(2a+a^2)\tr(u\transp X^2 u)+ (2a+a^2)^2(u\transp X u)\tr(u\transp Xu)\notag\\&=\tr(X^2) + 
2(2a+a^2)u\transp X^2 u + (2a+a^2)^2 (u\transp X u)^2.\notag
\end{align}
The steps above hold because  $\tr(AB)= \tr(BA), \; \tr(A+B) = \tr(A)+\tr(B),$ and $\tr(cA) = c\cdot\tr(A)$ for all $A,B\in\S^n$ and $c\in \R$.

Now observe that by construction $ uu\transp \preceq \dfrac{Z}{\|Z\|}$.  Thus using that $X\in L\cap \S^n_{++}$ and $\|X\|_F = 1$ we get \[
u\transp X u = X\bullet uu\transp \le  \frac{X\bullet Z}{\|Z\|} = \frac{P_LX\bullet Z}{\|Z\|}  = \frac{X\bullet P_LZ}{\|Z\|} \le  \frac{\|X\|_F \|(P_LZ)^+\|_F}{\|Z\|} \le \epsilon.
\] 
Furthermore, since $X \in \S^n_{++}$ and $\|X\|_F = 1$, it follows that $X-X^2 \in \S^n_{++}$ and so $u\transp X^2 u \le u\transp X u \le \epsilon.$  Hence \eqref{frob.norm} yields
\[\|D(X)\|_F^2 \le 1 + 2(2a+a^2)\epsilon + (2a+a^2)^2\epsilon^2 = (1+(2a+a^2)\epsilon)^2.\]
\qed
\end{proof}

The Rescaling and Projection Algorithm from Section~\ref{sec.chubanov}, namely Algorithm~\ref{alg:chubanov}, extends to Algorithm~\ref{alg:sdp}.

\medskip

\begin{algorithm}
  \caption{Projection and Rescaling Algorithm for Semidefinite Conic Systems
    \label{alg:sdp}}
  \begin{algorithmic}[1]
    \State  ({\bf Initialization})
    \Statex Let $P_L: \S^n \rightarrow \S^n$ be the orthogonal projection onto $L$. 
    \Statex Let $D:\S^n \rightarrow \S^n$ be the identity map, $P := P_L,$ and $a:=\sqrt{2}-1$.
	\State ({\bf Basic Procedure})
	\Statex 
	Find $Z\succneqq 0 $ such that either  $P(Z) \succ 0$ or $\|\left(P(Z)\right)^+\|_F\le \frac{1}{4n}\|Z\|$.
	\IIf{$P(Z) \succ 0 $} HALT and \Return $X = D^{-1}(P(Z)) \in L \cap \S^n_{++}$ \EndIIf 
  	\State ({\bf Rescaling step}) 
	\Statex Pick $u\in\R^n, \|u\|_2 = 1$ such that $Zu = \lambda_{\max}(Z) u$.
	\Statex Replace $D:\S^n \rightarrow \S^n$ with the mapping $X \mapsto (I+auu\transp)D(X)(I+auu\transp)$. 
	\Statex Let $P:=P_{D(L)}.$	
	\Statex Go back to step 2. 	
  \end{algorithmic}
\end{algorithm}

Theorem~\ref{thm.chubanov} and its proof readily extends as follows.

\begin{theorem} If $L \cap \S^n_{++} \ne \emptyset$ then Algorithm~\ref{alg:sdp} finds $X\in L \cap \S^n_{++} $ in at most $\log_{1.5}(1/\delta(L\cap \S^n_{++}))$ main iterations.
\end{theorem}

\medskip

As it was the case in Algorithm~\ref{alg:chubanov}, the bulk of the work in each main iteration of Algorithm~\ref{alg:sdp} is a call to the basic procedure. As we detail in Section~\ref{sec.basic.proc}, the  simplest implementations of the basic procedure are guaranteed to terminate in $\Oh(n^4)$ perceptron or von Neumann steps.  Each of these steps requires an operation of the form $Z \mapsto P(Z)$ in addition to a leading eigenvalue computation for a matrix in $\S^n$ and some other negligible computations.  Assuming that $P$ is maintained via an orthogonal basis for $D(L)$ each operation $Z \mapsto P(Z)$ requires $\Oh(mn^2)$ arithmetic operations where $m=\dim(L)$.  The operation $Z \mapsto P(Z)$ dominates the leading eigenvalue computation.  
Indeed, there are several methods from the numerical linear algebra literature (see, e.g.,~\cite{Parl80}) that compute the leading eigenvalue and eigenvector of an $n\times n$ symmetric matrix in $\Oh(n^2)$ arithmetic operations. It thus follows that the total number of arithmetic operations required by Algorithm~\ref{alg:sdp} is bounded above by 
$$\Oh(mn^2\cdot n^4\cdot \log(1/\delta(L\cap \S^n_{++}))) = \Oh(mn^6\log(1/\delta(L\cap \S^n_{++}))).$$

\medskip

Algorithm~\ref{alg:sdp}  extends in the same fashion as Algorithm~\ref{alg:chubanov} extends to Algorithm~\ref{alg:extended} to find a point in either $
L \cap \S^n_{++}$ or $L^{\perp} \cap \S^n_{++}$ provided one of them is feasible.

\section{Extension to symmetric conic systems}
\label{sec.symmetric}
Consider the problem 
\begin{equation}\label{sym.problem}
\text{ find } x \in L \cap \Omega, 
\end{equation}
where $L\subseteq V$ and $\Omega \subseteq V$ are respectively a linear subspace and  the interior of a symmetric cone in a finite-dimensional vector space $V$ over $\R$. 

We next present a version of the Projection and Rescaling Algorithm for the more general problem \eqref{sym.problem}.  To that end, we rely on some  machinery of Euclidean Jordan Algebras.  For succinctness we  recall only the essential facts and pieces of notation that are necessary for our exposition.  We refer the reader to the articles~\cite{SchmA01,SchmA03} and the textbooks~\cite{Baes09,FaraK94} for a more detailed discussion of Euclidean Jordan algebras and their connection to optimization.  The key connection between symmetric cones and Euclidean Jordan algebras is given by a theorem of Koecher and Vinberg that establishes a correspondence between symmetric cones and cones of squares of Euclidean Jordan algebras~\cite[Chapter III]{FaraK94}.

Assume $V$ is endowed with a bilinear operation $\circ : V\times V \rightarrow V$ and $e\in V$ is a particular element of $V$.  The triple $(V,\circ,e)$ is an {\em Euclidean Jordan algebra with identity element} if  the following conditions hold:
\begin{itemize}
\item $x\circ y = y \circ x$ for all $x,y\in V$
\item $x\circ(x^2 \circ y) = x^2\circ(x\circ y)$ for all $x,y\in V$, where $x^2 = x\circ x$
\item $x\circ e = x$ for all $x \in V$
\item There exists an associative positive definite bilinear form on $V$.
\end{itemize}
Example~\ref{the.example} below summarizes the most popular types of Euclidean Jordan algebras used in optimization.

An element $c\in V$ is {\em idempotent} if $c^2  = c$. An idempotent element of $V$ is a {\em primitive idempotent} if it is not the sum of two other idempotents.
The rank $r$ of $V$ is the smallest integer such that for all $x\in V$ the set $\{e,x,x^2,\dots,x^r\}$ is linearly dependent.  Every element $x\in V$ has a {\em spectral decomposition}
\begin{equation}\label{spectral}
x = \sum_{i=1}^r \lambda_i(x) c_i,
\end{equation}
where $\lambda_i(x) \in \R,\; i=1,\dots,r$ are the {\em eigenvalues} of $x$ and  $\{c_1,\dots,c_r\}$ is a {\em Jordan frame,} that is, a collection of non-zero primitive idempotents such that $c_i \circ c_j = 0$ for $i\ne j$, and $c_1+\cdots+c_r = e$.  We will rely on the following simple observation: Given the spectral decomposition \eqref{spectral}, we have $x\circ c_i = \lambda_i(x) c_i, \; i=1,\dots,r$.

The {\em trace} and {\em determinant} of $x \in V$ are respectively defined as
$
\tr(x) = \sum_{i=1}^r\lambda_i(x) 
$ and $\det(x) = \prod_{i=1}^r\lambda_i(x).$
Throughout this section we  assume that $(V,\circ,e)$ is an Euclidean Jordan algebra with identity.  Furthermore, we assume that $V$ is endowed with the following trace inner product:
\begin{equation}\label{inner.Jordan}
\ip{x}{y}:= \tr(x\circ y).
\end{equation}
We also assume that $\Omega$ is the interior of the cone of squares in $V,$ that is, $\Omega = \text{int}(\{x^2: x\in V\})$.  

\begin{example}\label{the.example} The following are the most popular Euclidean Jordan algebras used in optimization.  

\begin{enumerate}
\item[(a)] The space $\S^n$ of $n\times n$ real symmetric matrices with the bilinear operation
\[
X\circ Y:=\frac{XY+YX}{2}
\]
is an Euclidean Jordan algebra of rank $n$ and identity element $I$.  In this case, the spectral decomposition, trace, and determinant are precisely the usual ones.  The cone of squares is the cone of positive semidefinite matrices $\S^n_+$.
 \item[(b)]  The space $\R^n$ with the bilinear operation
 \[
 x\circ y = \matr{x_1y_1\\ \vdots \\ x_n y_n}
 \]
is an Euclidean Jordan algebra of rank n and identity element $e = \matr{1\\ \vdots \\ 1}$. In this case, the spectral decomposition of an element $x \in \R^n$ is 
\[
x = \sum_{i=1}^n x_i e_i.
\]
For $x\in \R^n$ we have $\tr(x) = \sum_{i=1}^n x_i$ and $\det(x) = \prod_{i=1}^n x_i$. The cone of squares is the non-negative orthant $\R^n_+$.
 \item[(c)]  
The space $\R^{n}$ with the bilinear operation
\[
\matr{x_0 \\ \bar x} \circ \matr{ y_0 \\ \bar y} := \matr{x\transp y\\ x_0 \bar y + y_0 \bar x}
\]
is an Euclidean Jordan algebra of rank 2 and identity element $e = \matr{1\\ \0}$.  In this case, the spectral decomposition of an element $x = \matr{x_0 \\ \bar x}\in \R^{n}$ is 
\[
x = (x_0 + \|\bar x\|_2) \matr{1/2 \\ \bar u/2} + (x_0 - \|\bar x\|_2) \matr{1/2 \\ -\bar u/2},
\]
where $ \bar u\in \R^{n-1}$ is such that $\|\bar u\|_2 = 1$ and  $\bar x = \|\bar x\|_2 \bar u$.  Consequently, for $x\in V$ we have $\tr(x) = 2x_0$ and $\det(x) = x_0^2 - \|\bar x\|^2$.
The cone of squares is the second order cone $\mathbb L_n = \left\{x = \matr{x_0 \\ \bar x}\in \R^{n}: x_0 \ge \|\bar x\|_2\right\}.$
 \item[(d)]  A direct product of finitely many of the above types of Euclidean Jordan algebras is again an Euclidean Jordan algebra.  
 \end{enumerate}
 
\end{example}

Let $P_L:V \rightarrow V$ be the projection map onto $L$ relative to the inner product defined in \eqref{inner.Jordan}.  Then \eqref{sym.problem} is equivalent to
\[
\text{ find } y \in V \text{ such that } P_Ly \in \Omega.
\]
We will rely on the Frobenius norm $\|x\|_F:=\sqrt{\ip{x}{x}} = \|\lambda(x)\|_2$, as well as on the operator norm $\|x\| := \|\lambda(x)\|_\infty$, 
where $\lambda(x)\in \R^r$ denote the vector of eigenvalues of $x\in V$. 

Consider the following kind of {\em condition measure} of the set $L \cap \Omega$:
\[
\delta(L \cap \Omega) := \dmax_x\left\{\det(x): x\in L \cap \Omega, \|x\|_F^2 = r\right\}. 
\]
Observe that this condition measure matches the ones defined in Section~\ref{sec.chubanov} and Section~\ref{sec.SDP} for the special cases $\Omega = \R^n_{++}$ and $\Omega = \S^n_{++}.$  As in those special cases, observe that $L\cap \Omega \ne \emptyset$ implies $\delta(L\cap \Omega) \in (0,1]$ with equality precisely when $e \in L\cap \Omega.$

Let $\bar \Omega$ denote the closure of $\Omega$.  For $v \in V$ let $v^+$ denote the projection of $v$ on $\bar \Omega$.  It is easy to see that if $v = \sum_{i=1}^r \lambda_i(v) c_i$ is the spectral decomposition of $v$, then $v^+ = \sum_{i=1}^r \lambda_i(v)^+ c_i$, where $\lambda_i(v)^+ = \max\{\lambda_i(v),0\}, \; i=1,\dots,r$.

Assume $c\in V$ is a primitive idempotent and $a > 0$ is a positive constant.  The following mapping associated to $c$ is key to our development.  Let $D_v :V\rightarrow V$ be the {\em quadratic mapping} associated to $v = e + a c$, that is,
\begin{equation}\label{Qx.eqn}
D_vx = 2v\circ(v\circ x) - v^2 \circ x.
\end{equation}
The following identities readily follow from the properties of the Jordan algebra product
\[
v \circ x = (e+ac)\circ x = x + a c\circ x
\]
\[
v \circ(v \circ x) = (e+ac)\circ (x + a c\circ x) = x + 2a c \circ x + a^2 c \circ (c\circ x)
\]
\[
v^2 \circ x =  (e + 2ac + a^2 c) \circ x =  x + (2a+a^2) c\circ x.
\]
Hence the quadratic mapping associated to $v = e + a c$ defined in \eqref{Qx.eqn} can also be written as
\begin{equation}\label{Qx.eqn.2}
D_vx = x+(2a-a^2)c\circ x + 2a^2 c\circ(c\circ x).
\end{equation}

\begin{proposition}\label{prop.socp}  Assume $z \in \bar\Omega\setminus \{0\}$ is such that $\|\left(P_Lz\right)^+\|_F \le \epsilon \|z\|.$ Let $c\in V$ be a primitive idempotent  such that $z\circ c = \lambda_{\max}(z) c$ and let $D_v:V\rightarrow V$ be the quadratic mapping associated to $v = e + a c$ as in \eqref{Qx.eqn} for some constant $a > 0$.  Then
\[
\delta(D_v(L) \cap \Omega) \ge \frac{(1+a)^2}{(1 + (2a + a^2)\epsilon)^r}\cdot \delta(L\cap \Omega).
\]
In particular, if  $\epsilon = \frac{1}{4r}$ and $a =\sqrt{2}-1$ then 
\[
\delta(D_v(L) \cap \Omega) \ge 1.5 \cdot \delta(L \cap \Omega).
\]
\end{proposition}
\begin{proof}
Observe that for $x \in L\cap \Omega$ the point $\hat x:=\frac{\sqrt{r}}{\|D_vx\|_F}\cdot D_vx$ satisfies $\hat x \in D_v(L) \cap \Omega$ and $\|\hat x\|_F^2 = r$.
Thus it suffices to show that for $x \in L\cap \Omega$ with $\|x\|_F^2 = r$ both $\det(D_vx) = (1+a)^2 \det(x)$ and $\|D_vx\|_F^2 \le r(1 + (2a + a^2)\epsilon)^2$
 as   this would imply $\det(\hat x) \ge \frac{(1+a)^2}{(1 + (2a + a^2)\epsilon)^r} \det(x).$  
Equivalently, it suffices to show that for $x \in L\cap \Omega$ with $\|x\|_F = 1$ both $\det(D_vx) \ge (1+a)^2 \det(x)$ and $\|D_vx\|_F \le 1 + (2a + a^2)\epsilon$. 
Assume $x \in L\cap \Omega$ with $\|x\|_F = 1$ is fixed. Since $D_v$ is the quadratic form associated to $v = e + a c$, it follows from~\cite[Prop III.4.2]{FaraK94} or from~\cite[Prop 2.5.4]{Baes09} that
\[
\det(D_vx) = \det(v)^2 \det(x) = (1+a)^2 \det(x).
\]
On the other hand, the expression \eqref{Qx.eqn.2} for $D_vx$ yields
\begin{equation}
\label{Qx.norm}
\|D_vx\|_F^2 = \|x\|_F^2 + 2(2a+a^2) \tr(c \circ x^2) + (2a+a^2)^2 \tr((c\circ x)^2).
\end{equation}
Next observe that $\frac{z}{\|z\|} - c \in \bar \Omega$.  Thus using that $x\in L\cap \Omega$ and $\|x\|_F = 1$ we get
\[
\ip{x}{c} \le \frac{1}{\|z\|}\ip{x}{z} =\frac{1}{\|z\|} \ip{P_Lx}{z} = \frac{1}{\|z\|} \ip{x}{P_Lz}
\le \frac{1}{\|z\|} \|x\|_F\|(P_Lz)^+\|_F \le  \epsilon.
\] 
Since $c,x\in \Omega$ we have $c\circ x \in \Omega$.  In particular, $\tr((c\circ x)^2) \le (\tr(c\circ x))^2 \le \epsilon^2$.  Furthermore, since $x\in \Omega$ and $\|x\|_F = 1$ we also have $x - x^2 \in \Omega$.   In particular $\tr(c\circ x^2) \le \tr(c \circ x) \le \epsilon$.
 Therefore \eqref{Qx.norm} yields
\begin{align*}
\|D_vx\|_F \le 1 + (2a + a^2) \epsilon.
\end{align*}
\qed
\end{proof}

We have the following more generic version of the Projection and Rescaling Algorithm presented in Algorithm~\ref{alg:symmetric}.

\bigskip

\begin{algorithm}
  \caption{Projection and Rescaling Algorithm for Symmetric Conic Systems
    \label{alg:symmetric}}
  \begin{algorithmic}[1]
    \State  ({\bf Initialization})
    \Statex Let $P_L: V \rightarrow V$ be the orthogonal projection onto $L$. 
    \Statex Let $D: V \rightarrow V$ be the identity map, $P := P_L,$ and $a:=\sqrt{2}-1$.    
	\State ({\bf Basic Procedure})
	\Statex  Find $z\in \bar\Omega\setminus \{0\}$ such that either $Pz \in \Omega$ or $\|(Pz)^+\|_F \le \frac{1}{4r}\|z\|$.	
	\IIf{$Pz \in \Omega$} HALT and \Return $x = D^{-1}Pz \in L \cap \Omega$ \EndIIf 
  	\State ({\bf Rescaling step}) 
	\Statex Pick $c\in V$ a primitive idempotent  point such that $z\circ c = \lambda_{\max}(z) c$.
	\Statex Let $D_v:V\rightarrow V$ be the quadratic mapping associated to $v = e + ac$.
	\Statex Replace $D$ with $D_vD$ and $P$ with $P_{D L}$.	
	\Statex Go back to step 2. 	
  \end{algorithmic} 
\end{algorithm}

\begin{theorem}\label{main.thm} If $L \cap \Omega \ne \emptyset$ then Algorithm~\ref{alg:symmetric} finds $x\in L \cap \Omega $ in at most $\log_{1.5}(1/\delta(L\cap \Omega))$ main iterations.
\end{theorem}

We note that in the special case when $V = \S^n$ with the  Euclidean Jordan algebra described in Example~\ref{the.example}(a), Algorithm~\ref{alg:symmetric} reduces to Algorithm~\ref{alg:sdp} in Section~\ref{sec.SDP}.  On the other hand, when $V = \R^n$ with the  Euclidean Jordan algebra described in Example~\ref{the.example}(b), Algorithm~\ref{alg:symmetric} yields a slightly weaker version of Algorithm~\ref{alg:chubanov} in Section~\ref{sec.chubanov}.  It is the small price we pay for extending the algorithm to general symmetric cones.  %The exact version from Section~\ref{sec.chubanov} can be recovered via a sharpening of Proposition~\ref{prop.socp} and of the above generic version of the projection and rescaling algorithm for direct products of Euclidean Jordan algebras.  

\medskip

Once again, the bulk of each main iteration in Algorithm~\ref{alg:symmetric} is a call to the basic procedure.  As we detail in Section~\ref{sec.basic.proc}, for $r$ = Jordan algebra rank of $V$ the simplest implementations of the basic procedure terminate in $\Oh(r^4)$ perceptron or von Neumann steps.  Each of these steps requires an operation of the form $z \mapsto P(z)$ in addition to a Jordan leading eigenvalue computation in $V$ and some negligible computations.  The amount of computational work required by the operation $z \mapsto P(z)$ dominates that of the other operations. 
Assuming that $P$ is maintained via an orthogonal basis for $D(L)$, %and assuming a bound $\Oh(r^3)$ on the number of arithmetic operations required by an eigenvalue computation in $V$, 
it follows that the total number of arithmetic operations required by Algorithm~\ref{alg:symmetric} is bounded above by 
$$\Oh(mn\cdot r^4\cdot\log(1/\delta(L\cap \Omega)))$$ 
where $m=\dim(L), \; n = \dim(V),$ and $r=$ Jordan algebra rank of $V$.

\medskip

Algorithm~\ref{alg:symmetric}  also extends in the same fashion as Algorithm~\ref{alg:chubanov} extends to Algorithm~\ref{alg:extended} to find a point in either $
L \cap \Omega$ or $L^{\perp} \cap \Omega$ provided one of them is feasible.

\section{The basic procedure}
\label{sec.basic.proc}
We next describe various possible implementations for the basic procedure, i.e., step 2 in the Projection and Rescaling Algorithm.  The schemes we discuss below vary in their work per iteration and overall speed of convergence.  Assume $\Omega$ and $V$ are as in Section~\ref{sec.symmetric} and define the {\em spectraplex}  $\Delta(\Omega)$ as follows:
\[
\Delta(\Omega):=\left\{x\in \bar\Omega: \ip{e}{x}= 1\right\}.
\]
Assume also that $P:V \rightarrow V$ is a projection mapping. 

\subsection{Perceptron scheme}

This is perhaps the simplest possible scheme.  It is based on the classical perceptron algorithm of Rosenblatt~\cite{Rosen}, which has a natural extension to conic systems~\cite{BellFV09,PenaS13}. We  assume that the following kind of separation oracle for $\Omega$ is an available:  Given $v\in V$, the separation oracle either verifies that $v\in \Omega$ or else it yields a separating vector $u\in \Delta(\Omega)$ such that $\ip{u}{v} \le 0$.

Observe that such a separation oracle is readily available when $\Omega$ is $\R^n_{++}, \S^n_{++},$ $\text{int}(\L_n)$ or any direct product of these kinds of cones.  Algorithm~\ref{alg:perceptron} gives an implementation of the basic procedure via the perceptron scheme.

\begin{algorithm}
  \caption{Perceptron Scheme
    \label{alg:perceptron}}
  \begin{algorithmic}[1]
  \State Pick $z_0 \in \Delta(\Omega)$ and $t:=0.$
   \While {$Pz_t \not\in \Omega$ and $\|(Pz_t)^+\|_F > \frac{1}{4r}\|z_t\|$}
\Statex Pick $u\in \Delta(\Omega)$ such that $\ip{u}{Pz_t} \le 0.$
\Statex Let $z_{t+1}:=\left(1-\frac{1}{t+1}\right)z_t + \frac{1}{t+1} u = 
\frac{t}{t+1}z_t + \frac{1}{t+1} u.$ 
\Statex $t :=t+1.$
\EndWhile

\end{algorithmic}
\end{algorithm}

\begin{proposition}\label{prop.perceptron} If Algorithm~\ref{alg:perceptron} has not halted after $t\ge 1$ iterations then 
\[
\|Pz_t\|_F^2 \le \frac{1}{t}.
\]
\end{proposition}
\begin{proof} Proceed by induction on $t$.  To that end, observe that $\|z\|_F \le 1$ for all $z\in \Delta(\Omega)$ and so $\|Pz\|_F \le 1$ since $P$ is a projection.  Therefore the condition readily holds for $t=1$.  Assume the condition holds for $t$ and the algorithm continues to iteration $t+1$.  Then
\begin{align*}
\|Pz_{t+1}\|_F^2 &= \frac{t^2}{(t+1)^2}\|Pz_t\|_F^2 + \frac{1}{(t+1)^2} \|Pu\|_F^2 + \frac{2t}{(t+1)^2} \ip{u}{Pz_t} \\
&\le \frac{t^2}{(t+1)^2}\|Pz_t\|_F^2 + \frac{1}{(t+1)^2} \|Pu\|_F^2 \\
& \le \frac{t^2}{(t+1)^2}\frac{1}{t} + \frac{1}{(t+1)^2} \\
& = \frac{1}{t+1}.
\end{align*}
\qed
\end{proof}

\begin{corollary} If the basic procedure is implemented via Algorithm~\ref{alg:perceptron}, then one of the stopping conditions $Pz \in \Omega$ or $\|(Pz)^+\|_F \le \frac{1}{4r}\|z\|$ is reached after at most $(4r^2)^2 = 16r^4$ iterations.  In the special case $\Omega = \R^n_+$ one of the stopping conditions $Pz > 0$ or $\|(Pz)^+\|_2 \le \frac{1}{3\sqrt{n}}\|z\|_\infty$ is reached after at most $(3n\sqrt{n})^2 = 9n^3$ iterations.
\end{corollary}
\begin{proof}
Both statements readily follow from Proposition~\ref{prop.perceptron} and the observations that $\|z\| \ge \frac{1}{r}$ for all $z\in \Delta(\Omega)$ and $\|v^+\|_F \le \|v\|_F$ for all $v\in V$. \qed
\end{proof}
\subsection{Von Neumann scheme}

The second scheme is based on a classical algorithm communicated by von Neumann to Dantzig and later studied by Dantzig in an unpublished manuscript~\cite{Dant92}. 

Several authors have studied various aspects of this algorithm over the last few years~\cite{EpelF02,PenaRS15,SoheP13}.
The von Neumann scheme can be seen as a greedy variation of the perceptron scheme that includes an exact line-search in each iteration. In the special case $\Omega = \R^n_+$, this scheme is essentially the same as the basic procedure proposed by Chubanov~\cite{Chub15}.

 Assume the following mapping $u:V \rightarrow \Delta(\Omega)$ is available:
\[
u(v):= \argmin_{u \in \Delta(\Omega)}\ip{u}{v}.
\]
Observe that such an mapping is readily available when $\Omega$ is $\R^n_{++}, \S^n_{++},$ $\text{int}(\L_n)$ or any direct product of these kinds of cones.  Algorithm~\ref{alg:vonNeumann} gives an implementation of the basic procedure via the von Neumann scheme.

\begin{algorithm}
  \caption{Von Neumann Scheme
    \label{alg:vonNeumann}}
  \begin{algorithmic}[1]
  \State Pick $z_0 \in \Delta(\Omega)$ and $t:=0$.
   \While {$Pz_t \not\in \Omega$ and $\|(Pz_t)^+\|_F > \frac{1}{4r}\|z_t\|$}
\Statex Let $u=u(Pz_t).$  
\Statex Let $z_{t+1}:=z_t + \theta_t (u-z_t)$ where
\[
\theta_t = \argmin_{\theta\in[0,1]} \|P(z_t + \theta (u-z_t))\|_F^2 = \frac{\|Pz_t\|_F^2 -\ip{u}{Pz_t}}{\|Pz_t\|_F^2 + \|Pu\|_F^2 - 2\ip{u}{Pz_t}}.
\] 
\Statex $t :=t+1.$
\EndWhile

\end{algorithmic}
\end{algorithm}

An inductive argument like the one used in the proof of Proposition~\ref{prop.perceptron} yields the following result.   However, we note that the choice of $u$ and $\theta_t$ at each iteration suggests that $\|Pz_t\|_F^2$ may decrease faster for this scheme than for the previous one.

\begin{proposition}\label{prop.vn} If Algorithm~\ref{alg:vonNeumann} has not halted after $t\ge 1$ iterations then 
\[
\|Pz_t\|_F^2 \le \frac{1}{t}.
\]
\end{proposition}

\begin{corollary}\label{corol.alg.vn} If the basic procedure is implemented via Algorithm~\ref{alg:vonNeumann}, then one of the stopping conditions $Pz \in \Omega$ or $\|(Pz)^+\|_F \le \frac{1}{4r}\|z\|$ is reached after at most $(4r^2)^2 = 16r^4$ iterations.  In the special case $\Omega = \R^n_+$ one of the stopping conditions $Pz > 0$ or $\|(Pz)^+\|_2 \le \frac{1}{3\sqrt{n}}\|z\|_\infty$ is reached after at most $(3n\sqrt{n})^2 = 9n^3$ iterations.
\end{corollary}
\begin{proof}
Both statements readily follow from Proposition~\ref{prop.vn} and the observations that $\|z\| \ge \frac{1}{r}$ for all $z\in \Delta(\Omega)$ and $\|v^+\|_F \le \|v\|_F$ for all $v\in V$. \qed
\end{proof}

\medskip

In the special case $\Omega = \R^n_+$ Corollary~\ref{corol.alg.vn}  recovers the iteration bound $\Oh(n^3)$ originally given by Chubanov~\cite[Lemma 2.2]{Chub15}.

\subsection{Smooth perceptron scheme}

Soheili and Pe\~na~\cite{SohPena, SoheP13} proposed a variation of the perceptron that relies on  the following tweaked version of the subproblem 
$\dmin_{u\in\Delta(\Omega)} \ip{u}{v}$ used in the von Neumann scheme.   Given $\mu > 0$ let $u_\mu : V \rightarrow \Delta(\Omega)$ be defined as
\[
u_\mu(v) := \argmin_{u\in\Delta(\Omega)}\left\{\ip{u}{v} + \frac{\mu}{2}\|u-\bar u\|^2\right\}.
\]
where $\bar u \in \Delta(\Omega)$ is a given point, e.g., $\bar u = \frac{1}{r} e.$

Pe\~na and Soheili~\cite{SoheP13} show that the mapping $u_\mu$ is readily available when $\Omega$ is $\R^n_{++}, \S^n_{++}, \L_n$ or any direct product of these kinds of cones.  Algorithm~\ref{alg:smooth} gives an implementation of the basic procedure via the smooth perceptron scheme.

\begin{algorithm}
  \caption{Smooth Perceptron Scheme
    \label{alg:smooth}}
  \begin{algorithmic}[1]
  \State Let $u_0 := \bar u$; $\mu_0 = 2$; $z_0 := u_{\mu_0}(Pu_0)$; and $t:=0.$
   \While {$Pu_t \not\in \Omega$ and  $\|(Pz_t)^+\|_F > \frac{1}{4r}\|z_t\|$}
\Statex  $\theta_t:=\frac{2}{t+3}$ 
\Statex   $u_{t+1}:=(1-\theta_t)(u_t + \theta_t z_t) + \theta_t^2 u_{\mu_t}(Pz_t)$
\Statex  $\mu_{t+1} = (1-\theta_t)\mu_t$
\Statex  $z_{t+1}:=(1-\theta_t)z_t + \theta_t u_{\mu_{t+1}}(Pu_{t+1})$ 
\Statex  $t :=t+1$.
\EndWhile

\end{algorithmic}
\end{algorithm}

\begin{proposition}\label{prop.smooth} If Algorithm~\ref{alg:smooth} has not halted after $t\ge 1$ iterations then 
\[
\|Pz_t\|_F^2 \le \frac{8}{(t+1)^2}.
\]
\end{proposition}

Proposition~\ref{prop.smooth} follows from \cite[Lemma 1]{SoheP13}.  For the sake of exposition, Lemma~\ref{the.lemma} below restates this technical result in the current context.  To that end, define $\varphi:  V\rightarrow \R$ as follows
\[
\varphi(z):= -\frac{1}{2}\|Pz\|_F^2 + \min_{u\in \Delta(\Omega)} \ip{u}{Pz}.
\]
Observe that $Pz \in \Omega$ if $\varphi(z)>0$. For $\mu > 0$ define $\varphi_\mu:  V\rightarrow \R$ as follows
\[
\varphi_\mu(z):= -\frac{1}{2}\|Pz\|_F^2 + \min_{u\in \Delta(\Omega)}\left\{\ip{u}{Pz} + \frac{\mu}{2}\|u-\bar u\|^2\right\}.
\]
\begin{lemma}[from {\cite{SoheP13}}]\label{the.lemma} 
\begin{description}
\item[(a)] For all $\mu >0$ and $z\in V$
\[
0 \le \varphi_\mu(z) - \varphi(z) \le \mu.
\]
\item[(b)] The iterates generated by Algorithm~\ref{alg:smooth} satisfy
\[
\frac{1}{2}\|Pz_t\|_F^2 \le \varphi_{\mu_t}(u_t).
\]
\end{description}
\end{lemma}

\noindent
{\bf Proof of Proposition~\ref{prop.smooth}.} Since  the algorithm has not halted after $t$ iterations we have $Pu_t \not \in \Omega$ and consequently $\varphi(u_t) \le 0$.  Thus Lemma~\ref{the.lemma} yields
\[
\|Pz_t\|_F^2 \le 2\varphi_{\mu_t}(u_t) \le 2(\mu_t+\varphi(u_t)) \le 2\mu_t.
\]
To conclude, observe that $\mu_t = \frac{4}{(t+1)(t+2)} \le \frac{4}{(t+1)^2}$.  
 \qed

\begin{corollary} If the basic procedure is implemented via Algorithm~\ref{alg:smooth}, then one of the stopping conditions $Pu \in \Omega$ or $\|Pz\|_F \le \frac{1}{4r}\|z\|$ is reached after at most $8\sqrt{2}r^2-1$ iterations. In the special case $\Omega = \R^n_+$ one of the stopping conditions $Pz > 0$ or $\|(Pz)^+\|_2 \le \frac{1}{3\sqrt{n}}\|z\|_\infty$ is reached after at most $6n\sqrt{2n} - 1$ iterations.
\end{corollary}
\begin{proof}
Both statements follow from Proposition~\ref{prop.smooth} and the observations that $\|z\| \ge \frac{1}{r}$ for all $z\in \Delta(\Omega)$ and $\|v^+\|_F \le \|v\|_F$ for all $v\in V$. 
\qed
\end{proof}

\medskip

The iteration bound $\Oh(r^2)$ for the smooth perceptron scheme versus the iteration bound $\Oh(r^4)$ for the perceptron scheme or von Neumann scheme does not account for the potentially higher cost of a smooth perceptron iteration.  Hence we next provide a bound on the number of arithmetic operations.  Aside from comparable operations of the form  $z\mapsto P(z)$, each iteration of the smooth perceptron requires the computation of $u_\mu(v)$ versus the computation of $u(v)$ required by the perceptron or von Neumann schemes.  As it is discussed in detail in~\cite{SoheP13}, the computation of $u_\mu(v)$ requires a complete eigenvalue decomposition of $v + \mu\bar u$ whereas the computation of $u(v)$ only requires computing the smallest eigenvalue and corresponding eigenvector of $v+\mu\bar u$.  By considering the special case of symmetric matrices, it follows that a complete Jordan eigenvalue decomposition requires $\Oh(r^3)$ arithmetic operations.  Hence the number of arithmetic operations required by the smooth perceptron scheme is bounded above by
\[
\Oh(\max(mn,r^3)\cdot r^2).
\]
On the other hand, a smallest eigenvalue calculation requires $\Oh(r^2)$ arithmetic operations.   Hence the number of arithmetic operations required by either the perceptron scheme or the von Neumann scheme is bounded above by
\[
\Oh(\max(mn,r^2)\cdot r^4).
\]
\subsection{Von Neumann with away steps scheme}

We now consider another variant on the von Neumann scheme that includes so-called {\em away steps}.  This can be seen as a particular case of the Frank-Wolfe algorithm with away steps that has recently become a subject of renewed attention~\cite{AhipST08,BeckS15,LacoJ15,PenaRS15}.  The away steps rely on the following construction. Given $z\in\Delta(\Omega)$, let $z = \sum_{i=1}^r \lambda_i(z) c_i$ be the spectral decomposition of $z$ and define the {\em support} of $z$ as $S(z):=\{c_i: \lambda_i(z) > 0\}.$ In principle we could update $z$ by decreasing the weight on an element of $S(z)$ while increasing the other weights.  Let 
\[c(z):=\argmax_{c\in S(z)} \ip{c}{Pz}
\]
and let $\lambda(z)$ denote the eigenvalue of $c(z)$ in the spectral decomposition of $z$.
Algorithm~\ref{alg:vonNeumann.away}
gives the implementation of the basic procedure via the von Neumann with away steps scheme.

%The idea of the variant is that the norm of $Pz$ can be decreased by decreasing the weight on $c(z)$.

\begin{algorithm}
  \caption{Von Neumann with Away Steps Scheme
    \label{alg:vonNeumann.away}}
  \begin{algorithmic}[1]
  \State Pick $z_0 \in \Delta(\Omega)$ and $t:=0$.
   \While {$Pz_t \not\in \Omega$ and $\|(Pz_t)^+\|_F > \frac{1}{4r}\|z_t\|$}
\Statex  Let $u = u(Pz_t)$ and $c = c(z_t).$
\Statex  {\bf if} $\|Pz_t\|^2 - \ip{u}{Pz_t} > \ip{c}{Pz_t} -  \|Pz_t\|^2$ {\bf then} (regular step)

\quad $a:= u-z_t; \; \theta_{\max} = 1$

\Statex {\bf else} (away step)

\quad $a:= z_t-c; \; \theta_{\max} = \frac{\lambda(z)}{1-\lambda(z)}$

\Statex{\bf endif}

\Statex Let $z_{t+1}:=z_t + \theta_t a$ where

\quad $
\theta_t = \displaystyle\argmin_{\theta\in[0,\theta_{\max}]} \|P(z_t + \theta a)\|_F^2 = \min\left\{ \theta_{\max} , -\frac{\ip{z_t}{Pa}}{\|Pa\|_F^2}\right\}.
$ 

\Statex  $t :=t+1$.
\EndWhile

\end{algorithmic}
\end{algorithm}

\begin{proposition}\label{prop.vna} If Algorithm~\ref{alg:vonNeumann.away} has not halted after $t\ge 1$ iterations then 
\[
\|Pz_t\|_F^2 \le \frac{8}{t}.
\]
\end{proposition}
\begin{proof} This readily follows via the same argument used in the proof of~\cite[Theorem 1(b)]{PenaRS15}. \qed
\end{proof}

\begin{corollary} If the basic procedure is implemented via Algorithm~\ref{alg:vonNeumann.away}, then one of the stopping conditions $Pz \in \Omega$ or $\|(Pz)^+\|_F \le \frac{1}{4r}\|z\|$ is reached after at most $8(4r^2)^2 = 128r^4$ iterations.  In the special case $\Omega = \R^n_+$ one of the stopping conditions $Pz > 0$ or $\|(Pz)^+\|_2 \le \frac{1}{3\sqrt{n}}\|z\|_\infty$ is reached after at most $8(3n\sqrt{n})^2 = 72n^3$ iterations.
\end{corollary}
\begin{proof}
Both statements readily follow from Proposition~\ref{prop.vna} and the observations that $\|z\| \ge \frac{1}{r}$ for all $z\in \Delta(\Omega)$ and $\|v^+\|_F \le \|v\|_F$ for all $v\in V$. \qed
\end{proof}

\medskip

We note that although the bound in Proposition~\ref{prop.vna} is  weaker than that in Proposition~\ref{prop.perceptron} and Proposition~\ref{prop.vn},  the von Neumann with away steps scheme tends to generate iterates $z \in \Delta(\Omega)$ with  smaller support.  Since these kinds of points in $\Delta(\Omega)$ in turn tend to have a larger value of $\|z\|$, this could be an advantage  as the scheme may reach the stopping condition $\|(Pz)^+\| \le \frac{1}{4r}\|z\|$ sooner.
  
\section{Updating the projection matrix}
\label{sec.projection}

Each rescaling step requires the update of the projection matrix from $P_L$ to $P_{D(L)}$.  We next describe how this update can be performed. As the subsections below detail, in certain important cases this update can be done much more efficiently than simply performing a naive recalculation of the projection matrix.

A possible approach to maintaining and updating the projection matrix is via orthogonal bases.  In particular, assume 
$P_L = QQ\transp$ for some matrix $Q\in \R^{n\times m}$ whose columns form an orthogonal basis of $L$, that is, $\text{span}(Q) = L$ and $Q\transp Q = I_m$.  To obtain a likewise expression 
$P_{D(L)} = \tilde Q \tilde Q\transp$ where the columns of $\tilde Q$ are an orthogonal basis of $D(L)$ we can proceed as follows.

First, observe that $\text{span}(DQ) = D(L)$.  Henceforth, it suffices to orthogonalize the columns of $DQ$.   That is, we need to find $R \in \R^{m\times m}$ such that $DQR$ is orthogonal, or equivalently such that
\begin{equation}\label{eqn.orth}
(DQR)\transp(DQR) = R\transp Q\transp D\transp D Q R = I_m.
\end{equation}
Although a matrix $R$ such that \eqref{eqn.orth} holds  could be achieved via a Gram-Schmidt procedure for the columns of $DQ$ or via a Cholesky factorization of $Q\transp D\transp D Q$, the particular structure of $D$ may enable a more efficient procedure.  In all of the cases discussed above $D$ is of the form $I_n+B$ for some structured and symmetric matrix $B\in \R^{n\times n}$.  In this case
\[
 Q\transp D\transp D Q = Q\transp(I_n+2B+B^2)Q = I_m + Q\transp (2B+B^2) Q.  
\]
Let $Q\transp (2B+B^2) Q = P\Lambda P\transp$ be the spectral decomposition of $Q\transp (2B+B^2) Q$ for some orthogonal matrix $P \in \R^{m\times p}$ and some diagonal matrix $\Lambda \in \R^{p\times p}$.  It readily follows that \eqref{eqn.orth} holds for
\[
R = I_m - P \bar \Lambda P\transp
\]
if $\bar \Lambda \in \R^{p\times p}$ is a diagonal matrix such that 
\begin{equation}\label{eq.diag}
(I_m - P \bar \Lambda P\transp)(I_m + P \Lambda P\transp)(I_m - P \bar \Lambda P\transp) = I_m.
\end{equation}
Observe that \eqref{eq.diag} holds provided the diagonal matrix $\bar \Lambda$ solves 
\[
- 2\bar \Lambda + \Lambda - 2\bar \Lambda \Lambda + \bar \Lambda^2 + \bar \Lambda^2 \Lambda = 0.
\]
One of the solutions of this equation is
\begin{equation}\label{eq.sol}
\bar \Lambda = (I_p+\Lambda)^{-1/2} + I_p.
\end{equation}
Notice that $\bar \Lambda$ is easily computable componentwise since $\Lambda$ is a diagonal matrix.

\subsection{The case $\Omega = \R^n_{++}$}
In this case $D$ is of the form $D = I + e_ie_i\transp$.  In this case $B = e_ie_i\transp$ and the term $Q\transp(2B+B^2)Q$ turns out to be
\[
3Q\transp e_ie_i\transp Q = 3 q_i q_i\transp
\] 
where $q_i = Q\transp e_i \in \R^m$.  The spectral decomposition of $Q\transp(2B+B^2)Q = 3 q_i q_i\transp$ is
\[
\frac{q_i}{\|q_i\|}\cdot (3\|q_i\|^2) \cdot \frac{q_i\transp}{\|q_i\|}.
\]
Hence  $$R = I_m - \frac{q_i}{\|q_i\|}\cdot \left(1+\frac{1}{\sqrt{1+3\|q_i\|^2}}\right) \cdot \frac{q_i\transp}{\|q_i\|} = I_m - \left(1+\frac{1}{\sqrt{1+3\|q_i\|^2}}\right) \cdot \frac{q_iq_i\transp}{\|q_i\|^2}
$$ and so
\[
\tilde Q = (I_n + e_ie_i\transp)Q\left(I_m - \left(1+\frac{1}{\sqrt{1+3\|q_i\|^2}}\right) \cdot \frac{q_iq_i\transp}{\|q_i\|^2}\right).
\]
\subsection{The case $\Omega = \S^n_{++}$}\label{sec.sdp.update}

Assume the ``columns'' of $Q$ correspond to the matrices $A_1,\dots,A_m \in \S^n$ such that 
$A_i \bullet A_i = 1,\; i=1,\dots,m$ and 
$A_i \bullet A_j =0, i\ne j$.  The columns of the new $DQ$ correspond to the matrices 
\[
D(A_i) = (I_n+auu\transp)A_i(I_n+a u u\transp) = A_i + a uu\transp A_i + a A_i uu\transp + 
a^2 (u\transp A_i u) u u\transp.
\]
In particular, the columns of the new $BQ = DQ - Q$ are 
\[
B(A_i) = a uu\transp A_i + a A_i uu\transp + 
a^2 (u\transp A_i u) u u\transp.
\]
Next, observe that
\[
B(A_i)\bullet A_j = 2a(A_iu)\transp(A_ju) + a^2(u\transp A_i u)(u\transp A_j u)
\]
and
\[
B(A_i)\bullet B(A_j) = 2a^2(A_iu)\transp(A_ju) + (2a^2+4a^3+a^4) (u\transp A_i u)(u\transp A_j u).
\]
Consequently, the $(i,j)$ entry %$2B(A_i)\bullet A_j + B(A_i)\bullet B(A_j)$ 
of the matrix $Q(2B+B^2)Q\transp$ is
\begin{align*}
2(2a+a^2)(A_iu)\transp(A_ju) + (2a+a^2)^2(u\transp A_i u)(u\transp A_j u).
\end{align*}
Therefore
\[
Q\transp(2B+B^2)Q = UWU\transp, 
\]
where
\begin{equation}\label{eqn.U.sdp}
U\transp = \matr{A_1 u & A_2 u & \cdots & A_m u \\ u\transp A_1 u & u\transp A_2 u & \cdots & u\transp A_m u}, \; W = \matr{2(2a+a^2) & 0 \\ 0&(2a+a^2)^2} .
\end{equation}

\medskip

When $m \le n$, it is typically cheaper to compute $R = L^{-\text{\rm T}}$ via the Cholesky factorization $LL\transp = I_m + UWU\transp$ of $Q\transp D\transp D Q = I_m + UWU\transp$.  
On the other hand, if $m \gg n+1$, it is typically more efficient to find the spectral decomposition $P\Lambda P\transp = UWU\transp$ for some orthogonal matrix $P\in\R^{m\times p}$  and some diagonal matrix $\Lambda \in \R^{p\times p}$, and then compute $R = I_m - P \bar \Lambda P\transp$ where $\bar \Lambda =  (I_p+\Lambda)^{-1/2} + I_p.$  In either case, it follows that the columns of $DQR$ form an orthogonal basis for $D(L)$.

\subsection{The case $\Omega = \text{int}(\L_{n})$}
\label{sec.socp.update}
In this case  $D$ is the matrix representation of the  mapping
\[
x \mapsto x + (2a-a^2) c\circ x + 2a^2 c\circ(c\circ x)
\]
where $c = \dfrac{1}{2}\matr{1 \\ \bar u}$, with $\bar u \in \R^{n-1}, \; \|\bar u\|_2 = 1.$
Observe that the mapping $x\mapsto c\circ x$ can be written as 
\[
x \mapsto \frac{1}{2}\matr{1 & \bar u\transp \\ \bar u & I}x.
\]
It thus follows that $D = I+B$ where
\[
B = \frac{2a-a^2}{2}\matr{1 & \bar u\transp \\ \bar u & I} + 
\frac{2a^2}{4}\matr{1 & \bar u\transp \\ \bar u & I}^2 
= a\matr{1 & \bar u\transp \\ \bar u & I} + 
\frac{a^2}{2} \matr{1 & \bar u\transp \\ \bar u & \bar u \bar u\transp}.
\]
Therefore,
\[
B^2 =a^2 \matr{1 & \bar u\transp \\ \bar u & I}
+ \left(a^2+2a^3 + \frac{a^4}{2}\right)\matr{1 & \bar u\transp \\ \bar u & \bar u \bar u\transp}
\]
and 
\begin{align}\label{eqn.U.socp}
2B+B^2 
= (2a+a^2)\left\{\matr{1 & \bar u\transp \\ \bar u & I}+\frac{2a+a^2}{2}\matr{1 & \bar u\transp \\ \bar u & \bar u \bar u\transp}\right\}.
\end{align}
In particular,
$
Q\transp(2B+B^2)Q$ is easily computable.  This computation provides the basis for the more interesting case when $\Omega$ is a direct product of semidefinite and second-order cones that we discussed next.

%As in Section~\ref{sec.sdp.update}, when $m \le n$ it is typically cheaper to compute $R = L^{-\text{\rm T}}$ via the Cholesky factorization $LL\transp = I_m + Q\transp(2B+B^2)Q$ of $Q\transp D\transp D Q = I_m + Q\transp(2B+B^2)Q$.  On the other hand, if $m \gg n+1$, it is typically more efficient to find the spectral decomposition $P\Lambda P\transp = Q\transp(2B+B^2)Q$ for some orthogonal matrix $P\in\R^{m\times p}$  and some diagonal matrix $\Lambda \in \R^{p\times p}$, and then compute $R = I_m - P \bar \Lambda P\transp$ where $\bar \Lambda =  (I_p+\Lambda)^{-1/2} + I_p.$  In either case, it follows that the columns of $DQR$ form an orthogonal basis for $D(L)$.

\subsection{Direct products of semidefinite and second-order cones}
We now consider the case  $\Omega = K_1\times \cdots\times K_r \subseteq \R^{n_1} \times \cdots \times \R^{n_r}$ where each $K_i \subseteq \R^{n_i}$ is a semidefinite cone or a second-order cone.  Assume $\R^{n_1} \times \cdots \times \R^{n_r}$ is endowed with the appropriate Euclidean Jordan algebra structure.  
It is easy to see that a primitive idempotent in this vector space is of the form $\matr{0 & \cdots& c_i\transp & \cdots & 0}\transp$ where $v_i$ is a primitive idempotent in $\R^{n_i}$.  It follows that the scaling matrix $D$ is of the form
\[
D = \matr{D_1 & \\ & \ddots \\ && D_r}
\]
where $D_j = I_{n_j}$ for $j\ne i$ and $D_i = I_{n_i} + B$ for some structured and symmetric matrix $B\in \R^{n_i \times n_i}$ that depends on the idempotent $c_i$.  
Observe that $Q\transp = \matr{Q_1\transp & \cdots & Q_r\transp}$ where each $Q_j\in \R^{n_j\times m}$.  It thus follows that
\[
Q\transp D\transp DQ = I_m + Q_i\transp(2B+B^2)Q_i. 
\]
The particular expression for the term $Q_i\transp(2B+B^2)Q_i$ is of the form \eqref{eqn.U.sdp} in Section~\ref{sec.sdp.update} or of the form $Q_i\transp(2B+B^2)Q_i$ where $2B+B^2$ is as in \eqref{eqn.U.socp} in Section~\ref{sec.socp.update}.
 
Again as we mentioned in Section~\ref{sec.sdp.update} and in Section~\ref{sec.socp.update}, when $m \le n_i$, it is typically cheaper to compute $R = L^{-1}$ via the Cholesky factorization $LL\transp =  I_m + Q_i\transp(2B+B^2)Q_i = Q\transp D\transp D Q$ whereas when $m \gg n_i$, it is typically more efficient to find the spectral decomposition $P\Lambda P\transp =  Q_i\transp(2B+B^2)Q_i$ and then compute $R = I_p - P \bar \Lambda P\transp$ where $\bar \Lambda =  (I_p+\Lambda)^{-1/2} + I_p.$  In either case it follows that $\tilde Q := DQR$ is an orthogonal basis of $D(L)$.

\section*{Acknowledgements}

Javier Pe\~na's research has been funded by NSF grant CMMI-1534850.

%\bibliographystyle{plain}
%\bibliography{references}

\end{document}